\documentclass{amsart}

\usepackage[all]{xy}

\theoremstyle{plain}
\newtheorem{thm}{Theorem}
\newtheorem{prop}[thm]{Proposition}

\theoremstyle{definition}

\theoremstyle{remark}

\newcommand{\HH}{\mathrm{H}}

\newcommand{\semid}{\unitlength.47cm
 \begin{picture}(.7,.6)
   \put(0,.05){$\times$}
   \put(.47,.04){\line(0,1){.39}}
 \end{picture}}
 
 \newcommand{\unter}[2]{\genfrac{}{}{0pt}{}{#1}{#2}}

\begin{document}

\title[Large universal deformation rings]{Large universal deformation rings}

\author{Frauke M. Bleher}
\address{Department of Mathematics\\University of Iowa\\
Iowa City, IA 52242-1419}
\email{frauke-bleher@uiowa.edu}
\thanks{The author was supported in part by  NSA Grant H98230-11-1-0131.}
\subjclass[2000]{Primary 20C20; Secondary 20C05, 16G10, 16G20}
\keywords{Universal deformation rings; defect groups;
generalized quaternion Sylow 2-subgroups; tame blocks; cyclic blocks}

\begin{abstract}
We provide a series of examples of finite groups $G$ and mod $p$ representations $V$ of
$G$ whose stable endomorphisms are all given by scalars
such that the universal deformation ring $R(G,V)$ of $V$ is large in the sense
that $R(G,V)/pR(G,V)$ is isomorphic to a power series algebra in one variable.
\end{abstract}

\maketitle

Let $k$ be a perfect field of positive characteristic $p$, let $\Gamma$ be a profinite group and 
let $V$ be a continuous finite dimensional representation of $\Gamma$ over $k$.
Suppose that the pro-$p$ completion of every open subgroup of $\Gamma$ is topologically
finitely generated (which is trivially satisfied when $\Gamma$ is a finite group).
It is an important representation theoretic
problem to determine when $V$ can be lifted to a representation over a local 
commutative ring of characteristic 0. If $W=W(k)$ is the ring of infinite Witt vectors over $k$, then
Green's lifting theorem shows, for example, that $V$ can be lifted to $W$ if 
$\HH^2(\Gamma,\mathrm{End}_k(V))=0$. The most natural generalization of such results is to
determine the full versal deformation ring $R(\Gamma,V)$ of $V$. 
The topological ring $R(\Gamma,V)$ is characterized by the property that 
the isomorphism class of every lift of $V$ over
a complete local commutative Noetherian ring $R$ with residue field $k$ arises
from a local ring homomorphism $\alpha: R(\Gamma,V)\to R$ and that $\alpha$ is unique
if $R$ is the ring of dual numbers $k[t]/(t^2)$. In case $\alpha$ is unique
for all $R$, $R(\Gamma,V)$ is called the universal deformation ring of $V$.
Note that all these rings $R$, including $R(\Gamma,V)$, have a natural $W$-algebra structure.
For more details on 
deformation rings and deformations, see for example \cite{mazur}. For the purpose of this paper,
\cite[Sect. 2]{3quat} provides the necessary background on deformation rings and 
deformations of representations of finite groups. 
It was shown in \cite[Prop. 2.1]{bc} that if $\Gamma$ is a finite group and
the stable endomorphism ring 
$\underline{\mathrm{End}}_{k\Gamma}(V)$ is isomorphic to $k$, then
the versal deformation ring $R(\Gamma,V)$ is always universal.

In this paper, we consider the question of how large $R(\Gamma,V)$ can be 
when $\Gamma$ is a finite group and $\underline{\mathrm{End}}_{k\Gamma}(V)\cong k$.
If $\Gamma$ is a profinite Galois group,  it is of interest to study the
case when $R(\Gamma,V)/pR(\Gamma,V)$ is finite dimensional over $k$, since this case may
lead to explicit presentations of $R(\Gamma,V)$
(see for example \cite{bockle} and its references). Since the representation $V$ 
factors through a finite quotient $G$ of $\Gamma$, it follows that if $R(G,V)/pR(G,V)$
is not finite dimensional over $k$ then $R(\Gamma,V)/pR(\Gamma,V)$ cannot be either.

For the remainder of this paper, suppose that $\Gamma=G$ is finite.
Let $V$ be a finitely generated indecomposable $kG$-module whose stable endomorphism ring is 
isomorphic to $k$ and which belongs to a block $B$ of $kG$ with a defect group $D$. 
It was shown in \cite{bc,bls} that if $D$ is cyclic, or if $p=2$, $D$ is dihedral and $B$ is Morita
equivalent to a principal block, then $R(G,V)$ is isomorphic to a subquotient ring of $WD$.
In particular,  $R(G,V)/pR(G,V)$ is  finite dimensional over
$k$ in these cases. Other instances where the finite dimensionality of $R(G,V)/pR(G,V)$
was established can be found for example in  \cite{3quat,bcdS} and their references.
However, de Smit and Rainone found examples in the case when $p\ge 5$
of finite groups $G$ and $kG$-modules $V$ such that $R(G,V)/pR(G,V)$ is isomorphic to $k[[t]]$,
and hence not finite dimensional over $k$ (see \cite[Remark 4.3]{bcdS} and \cite{bart}).

In this paper, we provide examples in Theorem \ref{thm:counter} (resp. Proposition
\ref{prop:oddp}) of finite groups $G$ and $kG$-modules $V$ such that 
$R(G,V)/pR(G,V)\cong k[[t]]$ when $p=2$ (resp. $p=3$).
In fact, the example in Proposition \ref{prop:oddp} works for all $p\ge 3$.
In particular, these examples provide a negative answer to \cite[Question 1.1]{bc} for 
all primes $p$.
The methods we use to construct our examples are different from the methods used
by de Smit and Rainone. Their computations do not involve block theory, whereas our 
methods to prove Theorem \ref{thm:counter} heavily rely on the 
description of tame blocks by Erdmann in \cite{erd} and the corresponding representation theory 
using quivers and relations. For the proof of Proposition \ref{prop:oddp}, we moreover use
the representation theory of cyclic blocks given by Brauer trees.

\begin{thm}
\label{thm:counter}
Suppose $k$ is an algebraically closed field of characteristic $p=2$.
Let $\overline{G}$ be a simple group with dihedral Sylow $2$-subgroups of order $2^d\ge 4$, 
and let $G$ be a non-trivial double cover of $\overline{G}$. 
Then there exists an indecomposable
$kG$-module $V$ of composition series length $4$ or $5$ which is a quotient module of  
a projective indecomposable $kG$-module such that $\underline{\mathrm{End}}_{kG}(V)\cong k$ 
and such that $R(G,V)\cong W[[t]]/(2f_V(t))$ for a certain power series $f_V(t)\in W[[t]]$. 
\end{thm}

\begin{proof}
Let $p=2$, and let $\overline{G}$, $d$ and $G$ be as in the statement of Theorem 
\ref{thm:counter}.
Let $\overline{B}$ (resp. $B$) be the principal block of $k\overline{G}$ (resp. $kG$). 
From the classification by Gorenstein and Walter of the groups with dihedral Sylow $2$-subgroups in \cite{gowa},
it follows that $\overline{G}$  is isomorphic to either $\mathrm{PSL}_2(\mathbb{F}_q)$ for some
odd prime power $q$ or to the alternating group $A_7$. Therefore, $G$ is isomorphic to either 
$\mathrm{SL}_2(\mathbb{F}_q)$ for some odd prime power $q$ or to the non-trivial double cover $\tilde{A}_7$.
In particular, it follows that both $\overline{B}$ and $B$ have precisely three isomorphism classes of 
simple modules.

The quivers and relations of the basic algebras of $\overline{B}$ and $B$ were determined in \cite{erd}.
In \cite{3quat}, the universal deformation rings were found of certain $kG$-modules that are 
inflated from $k\overline{G}$-modules belonging to $\overline{B}$. 
Since the notation introduced in \cite{3quat} is convenient 
for our purposes, we will freely use it. 
For the convenience of the reader, we will reproduce below
the quiver and relations of the basic algebra of $B$.

\smallskip

\textbf{Family (I) from \cite[Sect. 3.1]{3quat}.}
In family (I), $G\cong \mathrm{SL}_2(\mathbb{F}_q)$ where $q$ is a prime power with $q\equiv 1 \mod 4$ 
such that $2^{d+1}$ is the maximal $2$-power dividing $(q^2-1)$. The principal block $B$ of $kG$ is
Morita equivalent to $\Lambda=kQ/I$ where $Q$ and $I$ are as in Figure \ref{fig:familyI}.
The projective indecomposable $\Lambda$-modules are pictured for example in \cite[Fig. 1]{3quat}. 
\begin{figure}[ht] 
\caption{\label{fig:familyI} $\Lambda=kQ/I$ for blocks $B$ in family (I).}
$$
\xymatrix @R=-.2pc {
&0&\\
Q= \quad 1\; \bullet \ar@<.8ex>[r]^(.66){\beta} \ar@<1ex>[r];[]^(.34){\gamma}
& \bullet \ar@<.8ex>[r]^(.44){\delta} \ar@<1ex>[r];[]^(.56){\eta} & \bullet\; 2}
$$
\begin{eqnarray*}
I&=&\langle \beta\gamma\beta-\eta\delta\beta
(\gamma\eta\delta\beta)^{2^{d-1}-1},\gamma\beta\gamma-\gamma\eta\delta (\beta\gamma\eta
\delta)^{2^{d-1}-1}, \\ 
&&\eta\delta\eta-\beta\gamma\eta(\delta\beta\gamma\eta)^{2^{d-1}-1},
\delta\eta\delta-\delta\beta\gamma(\eta\delta\beta\gamma)^{2^{d-1}-1},\\ 
&&\delta\beta\gamma\beta,\gamma\eta\delta\eta\rangle.
\end{eqnarray*}
\end{figure}

Let $T$ be the 
$\Lambda$-module of $k$-dimension 5 which has an ordered $k$-basis given by (the images of)
$$(b_0,b_1,b_2,b_3,b_4)=(e_1,\beta,\gamma\beta,\delta\beta,\eta\delta\beta).$$
In particular, $T$ is a quotient module of the projective indecomposable
$\Lambda$-module $P_1=\Lambda e_1$ corresponding to the vertex $1$ in $Q$. 
The $\Lambda$-module structure of $T$ is given by the following $5\times 5$ matrices 
$X_c$ which describe the
action of (the image of) each vertex (resp. arrow) $c$ in $Q$ on $\{b_0,b_1,\ldots, b_4\}$:
$X_{e_0}=E_{11}+E_{44}$, $X_{e_1}=E_{00}+E_{22}$, $X_{e_2}=E_{33}$,
$X_\beta=E_{10}$, $X_\gamma=E_{21}$, 
$X_\delta=E_{31}$ and $X_\eta=E_{43}$. Here $E_{ji}$ denotes the $5\times 5$ matrix
whose $(j,i)$ entry is equal to 1  and all entries are equal to 0, i.e. $E_{ji}$ sends $b_i$ to $b_j$
and all other basis elements to 0.

Let $V$ be the $kG$-module which corresponds to $T$ under
the Morita equivalence between $B$ and $\Lambda$. Then $V$ is a quotient module of a
projective indecomposable $kG$-module.
Using the description of the projective indecomposable $\Lambda$-modules in
\cite[Fig. 1]{3quat}, it follows that 
$\underline{\mathrm{End}}_{\Lambda}(T)\cong k$, $\mathrm{Ext}^1_\Lambda(T,T)\cong k$
and $\mathrm{Ext}^2_\Lambda(T,T)\cong k$. Hence the Morita equivalence between $B$ and
$\Lambda$ also gives $\underline{\mathrm{End}}_{kG}(V)\cong k$, $\mathrm{Ext}^1_{kG}(V,V)\cong k$
and $\mathrm{Ext}^2_{kG}(V,V)\cong k$. By \cite[Sect. 1.6]{mazur}, it follows that the universal deformation ring
$R(G,V)$ is isomorphic to a quotient ring of $W[[t]]$ by an ideal generated by a single power series
(which could be zero). 

To finish the proof of Theorem \ref{thm:counter} for $G$ as in family (I), 
it is hence enough to show that the universal mod $2$
deformation ring $R(G,V)/2R(G,V)$ is isomorphic to $k[[t]]$. Note that $R(G,V)/2R(G,V)$ is
universal with respect to isomorphism classes of lifts of $V$ over complete
local commutative Noetherian $k$-algebras with residue field $k$.
Using the Morita equivalence between $B$ and $\Lambda$, it suffices to prove that
the universal deformation ring $R(\Lambda,T)$ of $T$ is isomorphic to $k[[t]]$, 
where $R(\Lambda,T)$ is universal
with respect to isomorphism classes of lifts of $T$  over complete
local commutative Noetherian $k$-algebras with residue field $k$ (see for example \cite[Prop. 2.5]{bv}).

Let $L$ be a free $k[[t]]$-module of rank 5 with $k[[t]]$-basis $\{B_0,B_1,\ldots, B_4\}$. 
Viewing $k$ as a subalgebra of $k[[t]]$, define a $\Lambda$-module structure on
$L$ as follows. 
Let (the image of) $e_i$, $i\in\{0,1,2\}$, (resp. $\zeta\in\{\beta,\delta,\eta\}$) act on $\{B_0,B_1,\ldots,B_4\}$ as 
the matrix $X_{e_i}$ (resp. $X_\zeta$), 
and let $\gamma$ act as $X_\gamma+tE_{24}$. Then $L$ is a $k[[t]]\otimes_k\Lambda$-module
which is free as a $k[[t]]$-module. Moreover, $L/tL\cong T$ as $\Lambda$-modules. 
Hence $L$ defines a lift of $T$
over $k[[t]]$. Therefore,  there exists a continuous $k$-algebra homomorphism 
$\varphi:R(\Lambda,T)\to k[[t]]$ corresponding to the isomorphism class of the lift 
of $T$ over $k[[t]]$ defined by $L$.
Since $L/t^2L$ defines a non-trivial lift of $T$ over $k[[t]]/(t^2)$, it follows that $\varphi$ is surjective.
Because $R(\Lambda,T)$ is isomorphic to a quotient algebra of $k[[t]]$, this implies that $\varphi$ is an 
isomorphism, proving Theorem \ref{thm:counter} for $G$ as in family (I).

\smallskip

\textbf{Family (II) from \cite[Sect. 3.2]{3quat}.}
In family (II), $G\cong \mathrm{SL}_2(\mathbb{F}_q)$ where $q$ is a prime power with $q\equiv 3 \mod 4$ 
such that $2^{d+1}$ is the maximal $2$-power dividing $(q^2-1)$. The principal block $B$ of $kG$ is
Morita equivalent to $\Lambda=kQ/I$ where $Q$ and $I$ are as in Figure \ref{fig:familyII}.
The projective indecomposable $\Lambda$-modules are pictured for example in \cite[Fig. 3]{3quat}. 
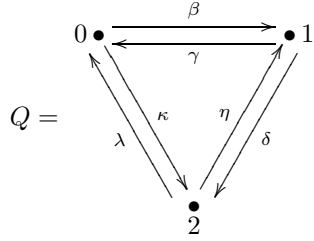
\begin{figure}[ht] 
\caption{\label{fig:familyII} $\Lambda=kQ/I$ for blocks $B$ in family (II).}
$$Q=\vcenter{\xymatrix  {
 0\,\bullet \ar@<.7ex>[rr]^{\beta} \ar@<.8ex>[rr];[]^{\gamma}\ar@<.7ex>[rdd]^{\kappa} \ar@<.8ex>[rdd];[]^{\lambda}
&&\bullet\ar@<.7ex>[ldd]^{\delta} \ar@<.8ex>[ldd];[]^{\eta}\,1\\&&\\ &
\unter{\mbox{\normalsize $\bullet$}}{\mbox{\normalsize $2$}}& }}$$
and 
\begin{eqnarray*}
I&=&\langle \delta\beta-\kappa\lambda\kappa,
\gamma\eta-\lambda\kappa\lambda,
\lambda\delta-\gamma\beta\gamma,
\eta\kappa-\beta\gamma\beta,\\
&&\beta\lambda-\eta(\delta\eta)^{2^{d-1}-1},\kappa\gamma-\delta(\eta\delta)^{2^{d-1}-1},\\
&&\delta\beta\gamma, \gamma\eta\delta, \eta\kappa\lambda
\rangle.
\end{eqnarray*}
\end{figure}

Let $T$ be the 
$\Lambda$-module of $k$-dimension 4 which has an ordered $k$-basis given by (the images of)
$$(b_0,b_1,b_2,b_3)=(e_0,\beta,\kappa,\lambda\kappa)$$
so that $T$ is a quotient module of the projective indecomposable
$\Lambda$-module $P_0=\Lambda e_0$. 
The $\Lambda$-module structure of $T$ is given by the following $4\times 4$
matrices with respect to the
$k$-basis $\{b_0,b_1,b_2, b_3\}$:
$Y_{e_0}=E_{00}+E_{33}$, $Y_{e_1}=E_{11}$, $Y_{e_2}=E_{22}$,
$Y_\beta=E_{10}$, $Y_\kappa=E_{20}$, $Y_\lambda=E_{32}$ and $Y_\gamma=0=Y_\delta=Y_\eta$.
Using the description of the projective indecomposable $\Lambda$-modules in
\cite[Fig. 3]{3quat}, it follows that 
$\underline{\mathrm{End}}_{\Lambda}(T)\cong k$ and $\mathrm{Ext}^1_\Lambda(T,T)\cong k
\cong\mathrm{Ext}^2_\Lambda(T,T)$. 
Let $L$ be a free $k[[t]]$-module of rank 4 with $k[[t]]$-basis $\{B_0,B_1,B_2, B_3\}$
and define a $\Lambda$-module structure on $L$ by letting
(the image of) $e_i$, $i\in\{0,1,2\}$, (resp. $\zeta\in\{\beta,\kappa,\lambda,\delta,\eta\}$) act on 
$\{B_0,B_1,B_2,B_3\}$ as the matrix $Y_{e_i}$ (resp. $Y_\zeta$), 
and $\gamma$ act as $Y_\gamma+tE_{31}$. Similarly to family (I), we conclude that
$R(\Lambda,T)\cong k[[t]]$.
If $V$ is the $kG$-module corresponding to $T$ under
the Morita equivalence between $B$ and $\Lambda$, then $V$ is a quotient module of a
projective indecomposable $kG$-module. We use similar arguments as in the proof for 
family (I) to show that $\underline{\mathrm{End}}_{kG}(V)\cong k$ and that the universal mod $2$
deformation ring of $V$  is isomorphic to $k[[t]]$, which proves Theorem
\ref{thm:counter} for $G$ as in family (II).

\smallskip

\textbf{Family (III) from \cite[Sect. 3.3]{3quat}.}
In family (III), $G$ is isomorphic to the non-trivial double cover $\tilde{A}_7$ and $d=3$.
The principal block $B$ of $kG$ is
Morita equivalent to $\Lambda=kQ/I$ where $Q$ and $I$ are as in Figure \ref{fig:familyIII}.
The projective indecomposable $\Lambda$-modules are pictured for example in \cite[Fig. 5]{3quat}. 
\begin{figure}[ht] 
\caption{\label{fig:familyIII} $\Lambda=kQ/I$ for blocks $B$ in family (III).}
$$\xymatrix @R=-.2pc {
&1&0&\\
 Q= \quad&\ar@(ul,dl)_{\alpha} \bullet \ar@<.8ex>[r]^{\beta} \ar@<.9ex>[r];[]^{\gamma}
& \bullet \ar@<.8ex>[r]^(.46){\delta} \ar@<.9ex>[r];[]^(.54){\eta} & \bullet\;2}$$
and
\begin{eqnarray*}
I&=&\langle \beta\alpha-\eta\delta\beta(\gamma\eta\delta\beta),
\alpha\gamma-\gamma\eta\delta(\beta\gamma\eta\delta), \\
&& \eta\delta\eta-\beta\gamma\eta(\delta\beta
\gamma\eta),\delta\eta\delta-\delta\beta\gamma(\eta\delta\beta\gamma)\rangle.
\end{eqnarray*}
\end{figure}

Let $T$ be the 
$\Lambda$-module of $k$-dimension 5 which has an ordered $k$-basis given by (the images of)
$$(b_0,b_1,b_2,b_3,b_4)=(e_2,\eta,\delta\eta,\gamma\eta,\beta\gamma\eta)$$
so that $T$ is a quotient module of the projective indecomposable
$\Lambda$-module $P_2=\Lambda e_2$.
The $\Lambda$-module structure of $T$ is given by the following $5\times 5$
matrices with respect to the
$k$-basis $\{b_0,b_1,\ldots,b_4\}$:
$Z_{e_0}=E_{11}+E_{44}$, $Z_{e_1}=E_{33}$, $Z_{e_2}=E_{00}+E_{22}$,
$Z_\beta=E_{43}$, $Z_\gamma=E_{31}$, $Z_\delta=E_{21}$, $Z_\eta=E_{10}$ and $Z_\alpha=0$.
Using the description of the projective indecomposable $\Lambda$-modules in
\cite[Fig. 5]{3quat}, it follows that 
$\underline{\mathrm{End}}_{\Lambda}(T)\cong k$ and $\mathrm{Ext}^1_\Lambda(T,T)\cong k
\cong\mathrm{Ext}^2_\Lambda(T,T)$. 
Let $L$ be a free $k[[t]]$-module of rank 5 with $k[[t]]$-basis $\{B_0,B_1,\ldots,B_4\}$
and define a $\Lambda$-module structure on $L$  by letting
(the image of) $e_i$, $i\in\{0,1,2\}$, (resp. $\zeta\in\{\alpha,\beta,\gamma,\eta\}$) act on 
$\{B_0,B_1,\ldots,B_4\}$ as the matrix $Z_{e_i}$ (resp. $Z_\zeta$), 
and $\delta$ act as $Z_\delta+tE_{24}$. Similarly to family (I), we conclude that
$R(\Lambda,T)\cong k[[t]]$. 
If $V$ is the $kG$-module corresponding to $T$ under
the Morita equivalence between $B$ and $\Lambda$, we use similar arguments as in the proof 
for family (I) to show that $V$ is a quotient module of a
projective indecomposable $kG$-module,
$\underline{\mathrm{End}}_{kG}(V)\cong k$ and  the universal mod $2$
deformation ring of $V$  is isomorphic to $k[[t]]$. This completes the proof of Theorem
\ref{thm:counter}. 
\end{proof}

The results in \cite{bart} together with Theorem \ref{thm:counter}  provide examples
when $$R(G,V)/pR(G,V)\cong k[[t]]$$ for all primes $p$ except $p=3$. 
The following result
provides an example for $p=3$, and in fact gives additional examples for  $p\ge 5$.
Note that $\mathbb{Z}_p=W(\mathbb{F}_p)$ denotes the ring of $p$-adic integers.

\begin{prop}
\label{prop:oddp}
Let $p\ge 3$, and define
$$G=(\mathbb{F}_p\times \mathbb{F}_p)\semid \mathbb{F}_p^{\,*}$$
where each $a\in\mathbb{F}_p^{\,*}$ acts on the $2$-dimensional vector space
$\mathbb{F}_p\times \mathbb{F}_p$ as multiplication by the diagonal
matrix $\mathrm{diag}(a,a^{-1})\in\mathrm{GL}_2(\mathbb{F}_p)$. 
Then there exists a uniserial $\mathbb{F}_pG$-module $V$ of composition series length
$p-1$ with $\mathrm{End}_{\mathbb{F}_pG}(V)\cong \mathbb{F}_p$ and 
$R(G,V)\cong \mathbb{Z}_p[[t]]/(pt)$.
\end{prop}

\begin{proof}
We write the elements of $G$ as pairs $\left((x,y),a\right)$,
where $(x,y)\in\mathbb{F}_p\times\mathbb{F}_p$ and $a\in \mathbb{F}_p^{\,*}$. 
Note that the action of $a$ on $(x,y)$ is given by $a.(x,y)=(ax,a^{-1}y)$. Let
$a_\varepsilon$ be a primitive element of $\mathbb{F}_p$, i.e. $a_\varepsilon$ generates the 
multiplicative group $\mathbb{F}_p^{\,*}$. Define
$\sigma=((1,0),1)$, $\tau=((0,1),1)$ and $\varepsilon=((0,0),a_\epsilon)$, and
let $K=\langle \tau\rangle$ and $\overline{G}=\langle \sigma,\varepsilon\rangle$.
Then $G=\langle \sigma,\tau,\varepsilon\rangle$, $K$ is a normal subgroup of
$G$, and $\overline{G}\cong G/K$. 
Moreover, $\overline{G}=\langle \sigma\rangle \semid \langle\varepsilon\rangle$
where $\varepsilon$ acts on $\sigma$ as multiplication by $a_\varepsilon$.
Hence  the simple 
$\mathbb{F}_p\overline{G}$-modules are all inflated from simple 
$\mathbb{F}_p\langle \varepsilon\rangle$-modules. If $T_i=\mathbb{F}_p$
such that $\varepsilon$ (resp. $\sigma$) acts as multiplication by $a_\varepsilon^i$ (resp. 1)
for $i=0,1,\ldots,p-2$, then $\{T_0,T_1,\ldots,T_{p-2}\}$ is a complete set of representatives of 
simple $\mathbb{F}_p\overline{G}$-modules. For $0\le i\le p-2$, the projective indecomposable 
$\mathbb{F}_p\overline{G}$-module cover $P_i$ of $T_i$ is uniserial of length $p$ with descending 
composition factors $T_i,T_{i+1},\ldots,T_{i+p-2},T_i$, where the indices are taken modulo $p-1$. 
Note that $\mathbb{F}_p\overline{G}$ is a Brauer tree algebra where the Brauer tree is a star 
with multiplicity 1 and $p-1$ edges labeled counter-clockwise in the order 
$T_0,T_1,\ldots,T_{p-2}$. Since the ring of $p$-adic integers
$\mathbb{Z}_p$ contains all $(p-1)^{\mathrm{th}}$ roots of unity,
it follows 
that its fraction field $\mathbb{Q}_p$ is a splitting field for $\overline{G}$.
The decomposition matrix of $\mathbb{F}_p\overline{G}$ is given in 
Figure \ref{fig:decomp}.

\begin{figure}[ht] 
\caption{\label{fig:decomp} The decomposition matrix of $\mathbb{F}_p\overline{G}$.}
$$\begin{array}{c@{}c}
&\begin{array}{c@{}c@{}c@{}c}\varphi_0\;&\;\varphi_1\;&\;\cdots\;&\;\varphi_{p-2}\end{array}\\[1ex]
\begin{array}{c}\chi_0\\ \chi_1\\ \vdots \\ \chi_{p-2}\\ \chi_{p-1}\end{array} &
\left[\begin{array}{cccc}1&0&\cdots&0\\0&\ddots&\ddots&\vdots\\ \vdots&\ddots&\ddots&0\\
0&\cdots&0&1\\1&1&\cdots&1\end{array}\right]
\end{array}$$
\end{figure}

Let $V$ be the uniserial $\mathbb{F}_p\overline{G}$-module with descending composition factors
$$T_0,T_1,\ldots, T_{p-2}.$$ 
In particular, $\mathrm{End}_{\mathbb{F}_p\overline{G}}(V)\cong \mathbb{F}_p$. 
An explicit matrix representation
$$\overline{\rho}:\overline{G}\to \mathrm{GL}_{p-1}(\mathbb{F}_p)$$
corresponding to $V$ is given in Figure \ref{fig:rep}.
Note that
\begin{equation}
\label{eq:needthis}
s_{p-2}=1\quad\mbox{and}\quad a_\varepsilon^{p-2}=a_\varepsilon^{-1} .
\end{equation}

\begin{figure}[ht] 
\caption{\label{fig:rep} The representation $\overline{\rho}:\overline{G}\to 
\mathrm{GL}_{p-1}(\mathbb{F}_p)$ corresponding to $V$.}
$$\begin{array}{ccc}
\overline{\rho}(\sigma)&=&\left(\begin{array}{ccccc}
1&s_1&s_2&\cdots&s_{p-2}\\
0&1&s_1&\ddots&\vdots\\ \vdots&\ddots&\ddots&\ddots&s_2\\
\vdots&&\ddots&1&s_1\\ 0&\cdots&\cdots&0&1
\end{array}\right)
\\[10ex]
\multicolumn{3}{c}{s_i^{-1}\;=\; i\,! \;\mbox{ in }\; \mathbb{F}_p^{\,*}\quad (1\le i\le p-2)\, ,}
\\[4ex]
\overline{\rho}(\varepsilon) &=& \left(\begin{array}{ccccc}
a_\varepsilon^{p-2}&0&\cdots&\cdots&0\\ 0&a_\varepsilon^{p-3}&0&&\vdots\\
\vdots&\ddots&\ddots&\ddots&\vdots\\ \vdots&&0&a_\varepsilon&0\\
0&\cdots&\cdots&0&1\end{array}\right).
\end{array}$$
\end{figure}

It follows from the description of the projective indecomposable 
$\mathbb{F}_p\overline{G}$-modules
that $\mathrm{Ext}^i_{\mathbb{F}_p\overline{G}}(V,V)=0$ for $i=1,2$. 
In particular, $V$ can be lifted over $\mathbb{Z}_p$ which implies
that $R(\overline{G},V)\cong \mathbb{Z}_p$.

Since $T_0$ is endo-trivial as an $\mathbb{F}_p\overline{G}$-module, it follows that
$V=\Omega^{-1}(T_0)$ is also endo-trivial as an $\mathbb{F}_p\overline{G}$-module.
In other words, there exists a projective $\mathbb{F}_p\overline{G}$-module 
$P$ such that $\mathrm{End}_{\mathbb{F}_p}(V)\cong T_0\oplus P$
as an $\mathbb{F}_p\overline{G}$-module. Using the explicit representation $\overline{\rho}$
in Figure \ref{fig:rep},  an easy matrix calculation shows that the socle
of the $\mathbb{F}_p\overline{G}$-module $\mathrm{End}_{\mathbb{F}_p}(V)$ is 
isomorphic to $T_0\oplus T_1\oplus \cdots \oplus T_{p-2}$. Hence
\begin{equation}
\label{eq:needthis2}
\mathrm{End}_{\mathbb{F}_p}(V)\cong T_0\oplus P_1\oplus\cdots\oplus P_{p-2}.
\end{equation}

We inflate $V$ from $\overline{G}$ to $G$, using that $\overline{G}\cong G/K$, 
and denote the resulting $\mathbb{F}_pG$-module also by $V$. 
Since $K$ acts trivially on $V$, and 
hence on  $\mathrm{End}_{\mathbb{F}_p}(V)$, it follows that 
\begin{eqnarray}
\label{eq:why}
\HH^i(G/K,\HH^0(K,\mathrm{End}_{\mathbb{F}_p}(V))
&\cong&\HH^i(\overline{G},\mathrm{End}_{\mathbb{F}_p}(V))\\
&\cong&
\mathrm{Ext}^i_{\mathbb{F}_p\overline{G}}(V,V)\;=\;0\qquad\mbox{ for $i=1,2$.}
\nonumber
\end{eqnarray}
Applying (\ref{eq:why}) to the Lyndon/Hochschild-Serre spectral sequence
$$\HH^p(G/K,\HH^q(K,\mathrm{End}_{\mathbb{F}_p}(V)))\Longrightarrow
\HH^{p+q}(G,\mathrm{End}_{\mathbb{F}_p}(V))$$
for $p+q=1$, we see that this spectral sequence degenerates to
$$\HH^1(G,\mathrm{End}_{\mathbb{F}_p}(V))
\cong \HH^0(G/K,\HH^1(K,\mathrm{End}_{\mathbb{F}_p}(V))).$$
Therefore we obtain isomorphisms
\begin{eqnarray*}
\HH^1(G,\mathrm{End}_{\mathbb{F}_p}(V)) &\cong& 
\HH^0(\overline{G},\HH^1(K,\mathrm{End}_{\mathbb{F}_p}(V))) \\
&\cong& \mathrm{Hom}(K,\mathrm{End}_{\mathbb{F}_p}(V))^{\overline{G}}\\
&\cong&\mathrm{Hom}_{\mathbb{F}_p\overline{G}}(T_{p-2},\mathrm{End}_{\mathbb{F}_p}(V))
\end{eqnarray*}
where the second isomorphism follows since $K$ acts trivially on $V$ and
the last isomorphism follows since $K=\langle \tau\rangle\cong\mathbb{F}_p$ and
$\varepsilon$ acts on $\tau$ as multiplication by $a_\varepsilon^{-1}=a_\varepsilon^{p-2}$. 
Hence it follows from (\ref{eq:needthis2}) that 
$\HH^1(G,\mathrm{End}_{\mathbb{F}_p}(V))\cong\mathbb{F}_p$, which implies that
the universal deformation ring $R(G,V)$ is isomorphic to a quotient ring of $\mathbb{Z}_p[[t]]$.

Since $R(\overline{G},V)\cong \mathbb{Z}_p$, there exists a matrix representation
$\rho_{\overline{G}}:\overline{G}\to\mathrm{GL}_{p-1}(\mathbb{Z}_p)$
such that the reduction of  $\rho_{\overline{G}}$ modulo $p$ is equal to the representation
$\overline{\rho}$ in Figure \ref{fig:rep}. Define $E$ to be the $(p-1)\times (p-1)$ matrix whose
last entry in the first row is equal to 1 and whose all other entries are equal to 0, and let
$I_{p-1}$ denote the identity matrix of size $p-1$. Let $R=\mathbb{Z}_p[[t]]/(pt)$
and define
$$\rho:G\to \mathrm{GL}_{p-1}(R)$$
by $\rho(g)=\rho_{\overline{G}}(g)$ for all $g\in\langle \sigma,\varepsilon\rangle$ and
$\rho(\tau)=I_{p-1}+tE$.
Using (\ref{eq:needthis}) and that $pt=0$ in $R$, it follows 
that $\rho$ is a group representation of $G$ which
defines a lift of $V$ over $R$ when $V$ is viewed as an $\mathbb{F}_pG$-module.

By the universal property of  $R(G,V)$, there exists a 
unique continuous $\mathbb{Z}_p$-algebra homomorphism
$\gamma:R(G,V)\to R$  corresponding to the isomorphism class of the lift $\rho$. 
If $\mathfrak{m}$ (resp. $\mathfrak{m}_R$) denotes the maximal ideal of $R(G,V)$
(resp. $R$), then $\gamma$ is surjective if and only if it induces a surjection 
\begin{equation}
\label{eq:tan}
\overline{\gamma}:\quad
\frac{R(G,V)}{\mathfrak{m}^2 + pR(G,V)} 
\longrightarrow \frac{R}{\mathfrak{m}^2_R + pR} \cong \frac{\mathbb{F}_p[t]}{(t^2)}.
\end{equation}
Let $\pi_u:R(G,V)\to R(G,V)/(\mathfrak{m}^2 + pR(G,V))$ and
$\pi:R\to R/( \mathfrak{m}^2_R + pR) \cong \mathbb{F}_p[t]/(t^2)$ be  the natural surjections.
Then $\overline{\gamma}\circ \pi_u$ defines a lift of $V$ over $\mathbb{F}_p[t]/(t^2)$,
and $\overline{\gamma}$ is surjective if and only if this lift is not isomorphic to the
trivial lift of $V$ over $\mathbb{F}_p[t]/(t^2)$. However,
since $\overline{\gamma}\circ \pi_u=\pi\circ \gamma$, the lift defined by
$\overline{\gamma}\circ \pi_u$ is isomorphic to
the lift defined by the reduction of $\rho$ modulo $\mathfrak{m}^2_R + pR$.
Since the latter lift is not isomorphic to the trivial lift  of $V$ over $\mathbb{F}_p[t]/(t^2)$,
it therefore follows that $\overline{\gamma}$, and hence $\gamma$, is surjective.
Since $R(G,V)$ is isomorphic to a quotient ring of $\mathbb{Z}_p[[t]]$, 
(\ref{eq:tan}) is in fact an isomorphism.
Let $r$ be any element of $R(G,V)$ such that $\gamma(r)$ is the class of
$t$ in $R = \mathbb{Z}_p[[t]]/(pt)$.  We then have a unique continuous 
$\mathbb{Z}_p$-algebra homomorphism
$$\mu:\mathbb{Z}_p[[t]] \to R(G,V) $$
which maps $t$ to $r$.  
Since $(\gamma\circ\mu)(t)$ is the class of $t$ in $R$, we see that $\gamma\circ\mu$ is surjective.
So because $\overline{\gamma}$ is an isomorphism, Nakayama's lemma implies
that $\mu:\mathbb{Z}_p[[t]] \to R(G,V)$ is surjective. Since $(\gamma\circ\mu)(t)=t$, it follows that the
kernel of $\mu$ is contained in $(pt)$.

Suppose $\gamma$ is not an isomorphism. Then the kernel of $\mu$ is
properly contained in $(pt)$, which means that
$$\mathrm{Ker}(\mu)\subseteq (pt)\cdot (p,t) = (p^2t,pt^2).$$
Hence there exists a surjective continuous $\mathbb{Z}_p$-algebra homomorphism
$\lambda: R(G,V)\to \mathbb{Z}_p[[t]]/(p^2t,pt^2)$ such that the composition of $\lambda$
with the natural surjection 
$\nu:\mathbb{Z}_p[[t]]/(p^2t,pt^2)\to \mathbb{Z}_p[[t]]/(pt)=R$
is equal to $\gamma$. This implies that there exists a group homomorphism
$\hat{\rho}:G\to\mathrm{GL}_{p-1}(\mathbb{Z}_p[[t]]/(p^2t,pt^2))$ with
$\nu\circ\hat{\rho}=\rho$.
In particular,
$$\hat{\rho}(\tau) = I_{p-1}+tE + ptA$$
for some matrix $A\in\mathrm{Mat}_{p-1}(\mathbb{Z}_p[[t]])$ where we view this equation
modulo the ideal $(p^2t,pt^2)$. Since $\hat{\rho}$ is a group homomorphism, 
$\hat{\rho}(\tau)^p$ must be equal to the identity $I_{p-1}$ modulo $(p^2t,pt^2)$. However,
\begin{eqnarray*}
\hat{\rho}(\tau)^p&\equiv&(I_{p-1}+tE + ptA)^p\\
&\equiv&I_{p-1}+ptE+p^2tA\\
&\equiv& I_{p-1} + ptE \mod (p^2t,pt^2)
\end{eqnarray*}
which means that 
$\hat{\rho}$ does not exist.
Therefore, $\gamma$ must be an isomorphism, implying that
$R(G,V)\cong R=\mathbb{Z}_p[[t]]/(pt)$.
\end{proof}

\end{document}